\documentclass[12pt,reqno]{amsart}
\usepackage{bbm}
\usepackage{}
\usepackage{stmaryrd}
\usepackage{mathrsfs}
\usepackage{cases}
\usepackage{amsfonts}
\usepackage{amssymb}
\usepackage{amsmath}
\usepackage{tikz}
\usepackage{extarrows}
\allowdisplaybreaks[4]
\newtheorem{theorem}{Theorem}[section]
\newtheorem{lemma}[theorem]{Lemma}

\newtheorem{corollary}[theorem]{Corollary}
\theoremstyle{definition}
\newtheorem{definition}[theorem]{Definition}

\numberwithin{equation}{section}

\let\d=\delta

\let\ep=\epsilon
\let\va=\varphi

\let\fy=\infty
\def\ejx{E_{j,x}}
\def\ejx1{E_{j,x}^1}
\def\ejx2{E_{j,x}^2}

\newcommand{\be}{\begin{equation*}}
\newcommand{\ee}{\end{equation*}}
\newcommand{\ben}{\begin{equation}}
\newcommand{\een}{\end{equation}}
\newcommand{\bn}{\begin{enumerate}}
\newcommand{\en}{\end{enumerate}}
\newcommand{\bs}{\backslash}
\newcommand{\ba}{\begin{align}}
\newcommand{\ea}{\end{align}}

\def\rr{{\mathbb R}}
\def\rn{{{\rr}^n}}

\begin{document}
\title[The unboundedness of Hausdorff operators on Quasi-Banach spaces]
{The unboundedness of Hausdorff operators on Quasi-Banach spaces}

\author{WEICHAO GUO}
\address{School of Science, Jimei University, Xiamen, 361021, P.R.China}
\email{weichaoguomath@gmail.com}
\author{JIANMIAO RUAN}
\address{Department of Mathematics, Zhejiang International Studies University,
Hangzhou 310012, P. R. China}
\email{rjmath@163.com}

\author{GUOPING ZHAO}
\address{School of Applied Mathematics, Xiamen University of Technology, Xiamen, 361024, P.R.China}
\email{guopingzhaomath@gmail.com}
\thanks{}

\subjclass[2010]{Primary 42B35; Secondary 47G10.}

\keywords{Hausdorff operator, local Hardy space, modulation space, unboundedness.}

\begin{abstract}
  In this note, we show that the Hausdorff operator $H_{\Phi}$
  is unbounded on a large family of Quasi-Banach spaces, unless $H_{\Phi}$ is a zero operator.
\end{abstract} \maketitle

\section{Introduction and Motivation}
Hausdorff operator, in connection with some classical summation methods, has been studied for a long time in the field of real and complex analysis. For the historical background and recent developments of the boundedness on function spaces regarding Hausdorff operators,
we refer the reader to survey papers by Chen-Fan-Wang \cite{ChenFanWang2013AMJCUSB} and Liflyand \cite{Liflyand2011}.

For a suitable function $\Phi$, the corresponding Hausdorff operator $H_{\Phi}$ can be defined by
\begin{equation}
  H_{\Phi}f(x):=\int_{\mathbb{R}^n}\Phi(y)f\left(\frac{x}{|y|}\right)dy.
\end{equation}

Particularly, when $\Phi$ is taken suitably, Hausdorff operator contains some important operators in the field of harmonic analysis.
For instance, the Hardy operator, adjoint Hardy operator (see \cite{ChenFanLi2012CAMSB, ChenFanZhang2012AMSES, FanLin2014AB}),
and the Ces\`{a}ro operator \cite{Miyachi2004JFAA, Siskakis1990PAMS} in one dimension.
The Hardy-Littlewood-P\'{o}lya operator and the Riemann-Liouville fractional integral operator can also be derived from the Hausdorff operator.

In recent years, there is an increasing interest on
the boundedness of Hausdorff operators on function spaces,
one can see \cite{GaoWuGuo2015MIA,GaoZhong2014MIA} for the
the boundedness on Lebesgue spaces,
\cite{ZhaoFanGuo2018AFA} for the boundedness on modulation and Wiener amalgam spaces,
and \cite{FanLin2014AB, LiflyandMoricz2000PAMS, RuanFan2016JMAA} for the boundedness on $H^1$ and $h^1$.
However, since the argument of using Minkowski's inequality cannot be applied to Quasi-Banach spaces,
there are only a few boundedness results considering the boundedness of Hausdorff operators on Quasi-Banach spaces.

Among the previous results, the study of Hausdorff operators on $H^p(\rn)\ (0<p<1)$ has its special status, since
$H^p(\rn)$ is a Quasi-Banach space and also an important function space in the field of harmonic analysis.
The study of $H_{\Phi}f$ on $H^p(\rr)$ was first initiated by Kanjin \cite{Kanjin2001SM} and continued in
\cite{LiflyandMiyachi2009SM}.
Very recently, Liflyand-Miyachi \cite{LiflyandMiyachi2019TAMS} establishes the multidimensional boundedness results on $H^p(\rn)$.
For the $H^p(\rr)$ boundedness of another type of Hausdorff operator,
we refer the reader to \cite{ChenFanLinRuan2016AM, RuanFan2017MN}.

Based on the previous results,
an interesting problem is whether we can establish the boundedness result on local Hardy space $h^p(\rn)$,
or even on the general inhomogeneous frequency decomposition spaces such as  Triebel-Lizorkin spaces $F^{p,q}_s(\rn)$,
Besov spaces $B^{p,q}_s(\rn)$ and modulation spaces $M^{p,q}_s(\rn)$ with $p\in (0,1)$. In this paper, we will explore this problem and give an negative answer,
showing that the Hausdorff operator is unboundedness on a large family of
Quasi-Banach spaces.
In particular, we prove that all the nonzero Hausdorff operators are unbounded on
$F_s^{p,q}$, $B_s^{p,q}$ or $M^{p,q}_s$ with $p\in (0,1)$.

Our paper is organized as follows.
In Section 2, we collect some notations and basic properties of function spaces.
The unboundedness results and their proofs will be presented in Section 3.

Throughout this paper, we will adopt the following notations.
We use $X\lesssim Y$ to denote the statement that $X\leq CY$, with a positive constant $C$ that may depend on $n, \,p$,
but it might be different from line to line.
The notation $X\sim Y$ means the statement $X\lesssim Y\lesssim X$.
We use $X\lesssim_{\lambda}Y$ to denote $X\leq C_{\lambda}Y$,
meaning that the implied constant $C_{\lambda}$ depends on the parameter $\lambda$.
For $x=(x_{1},x_{2},...,x_{n})\in \rn$,
we denote $|x|_{\infty }:=\max\limits_{i=1,2,...,n}|x_{i}|$ and $\langle x\rangle:=(1+|x|^{2})^{{1}/{2}}.$

\section{Preliminary}
Let $\mathscr{S}:= \mathscr{S}(\mathbb{R}^{n})$ be the Schwartz space
and $\mathscr{S}':=\mathscr{S}'(\mathbb{R}^{n})$ be the space of tempered distributions.
We define the Fourier transform $\mathscr {F}f$ and the inverse Fourier transform $\mathscr {F}^{-1}f$ of $f\in \mathscr{S}(\mathbb{R}^{n})$ by
$$
\mathscr {F}f(\xi)=\hat{f}(\xi)=\int_{\mathbb{R}^{n}}f(x)e^{-2\pi ix\cdot \xi}dx
,
~~
\mathscr {F}^{-1}f(x)=\check{f}(x)=\int_{\mathbb{R}^{n}}f(\xi)e^{2\pi ix\cdot \xi}d\xi.
$$

The local Hardy space was introduced by Goldberg \cite{Goldberg1979DMJ}.
Let $0<p<\infty$ and let $\psi\in\mathscr{S}$ satisfy $\int_{\mathbb{R}^n}\psi(x)dx\neq0$.
Define $\psi_t(x): =t^{-n}\psi(x/t)$.
The \emph{local Hardy space} is defined by
\begin{equation*}
  h^p:=\{f\in\mathscr{S}': \|f\|_{h^p}=\|\sup_{0<t<1}|\psi_t\ast f|\|_{L^p}<\infty\}.
\end{equation*}
We note that the definition of the local Hardy spaces is independent of the choice of $\psi\in \mathscr{S}$.

To define Besov space and Triebel-Lizorkin space, we introduce  the dyadic decomposition of $\mathbb{R}^n$.
Let $\varphi$ be a smooth bump function supported in the ball $\{\xi: |\xi|<3/2\}$ and be equal to 1 on the ball $\{\xi: |\xi|\leq 4/3\}$.
Denote
\begin{equation}\label{decompositon function of Besov space}
\phi(\xi)=\varphi(\xi)-\varphi(2\xi),
\end{equation}
and a function sequence
\begin{equation}
\begin{cases}
\phi_j(\xi)=\phi(2^{-j}\xi),~j\in \mathbb{Z}^{+},
\\
\phi_0(\xi)=1-\sum\limits_{j\in \mathbb{Z}^+}\phi_j(\xi)=\varphi(\xi).
\end{cases}
\end{equation}
For integers $j\geq 0$, we define the Littlewood-Paley operators
\begin{equation}
\Delta_j=\mathscr{F}^{-1}\phi_j\mathscr{F}.
\end{equation}
Let $0<p,q\leq \infty$, $s\in \mathbb{R}$. For a tempered distribution  $\ f$,  we set the norm
\begin{equation}
\|f\|_{B^{p,q}_s}=\left(\sum_{j=0}^{\infty}2^{jsq}\|\Delta_jf\|_{L^p}^q \right)^{1/q},~
\end{equation}
with the usual modifications when $q=\infty$.
The (inhomogeneous) Besov  space ${B^{p,q}_s}$ is the space of all tempered distributions $f$ for which the quantity $\|f\|_{B^{p,q}_s}$ is finite.
Let $0<p<\infty$, $0<q\leq \infty$, $s\in \mathbb{R}$. For a tempered distribution  $\ f$,  we set the norm
\begin{equation}
\|f\|_{F^{p,q}_s}=\left\|\left(\sum_{j=0}^{\infty}2^{jsq}|\Delta_jf|^q \right)^{1/q}\right\|_{L^p}
\end{equation}
with the usual modifications when $q=\infty$.
The Triebel-Lizorkin space ${F^{p,q}_s}$ is the space of all tempered distributions $f$ for which the quantity $\|f\|_{F^{p,q}_s}$ is finite.
We recall that the local Hardy space $h^p$ is equivalent with the inhomogeneous Triebel-Lizorkin space $F^{p,2}_0$ for $p\in(0,\infty)$

Next, we introduce the modulation space.
Denote by $Q_{k}$ the unit cube with the center at $k$. Then the family $\{Q_{k}\}_{k\in\mathbb{Z}^{n}}$
constitutes a decomposition of $\mathbb{R}^{n}$.
Let
$\eta: \mathbb{R}^{n} \rightarrow [0,1]$ be a smooth function satisfying that $\eta(\xi)=1$ for
$|\xi|_{\infty}\leq {1}/{2}$ and $\eta(\xi)=0$ for $|\xi|_{\infty}\geq 3/4$. Let
\begin{equation*}
\eta_{k}(\xi)=\eta(\xi-k),  k\in \mathbb{Z}^{n}
\end{equation*}
be a translation of \ $\eta$.
Since $\eta_{k}(\xi)=1$ in $Q_{k}$, we have that $\sum_{k\in\mathbb{Z}^{n}}\eta_{k}(\xi)\geq 1$
for all $\xi\in\mathbb{R}^{n}$. Denote
\begin{equation*}
\sigma_{k}(\xi)=\eta_{k}(\xi)\left(\sum_{l\in\mathbb{Z}^{n}}\eta_{l}(\xi)\right)^{-1},  ~~~~ k\in\mathbb{Z}^{n}.
\end{equation*}
It is easy to know that $\{\sigma_{k}\}_{k\in\mathbb{Z}^{n}}$
constitutes a smooth partition of the unity, and $%
\sigma_{k}(\xi)=\sigma(\xi-k)$. The frequency-uniform decomposition
operators can be defined by
\begin{equation*}
\Box_{k}:= \mathscr{F}^{-1}\sigma_{k}\mathscr{F}
\end{equation*}
for $k\in \mathbb{Z}^{n}$.
Now, we give the (discrete) definition of modulation space $M^{p,q}_s$.

\begin{definition}
Let $s\in \mathbb{R}, 0<p,q\leq \infty$. The modulation space $M^{p,q}_s$ consists of all $f\in \mathscr{S}'$ such that the (quasi-)norm
\begin{equation*}
\|f\|_{M^{p,q}_s}:=\left( \sum_{k\in \mathbb{Z}^{n}}\langle k\rangle^{sq}\|\Box_k f\|_{L^p}^{q}\right)^{1/q}
\end{equation*}
is finite.
Note that this definition is independent of the choice of $\{\sigma_k\}_{k\in\mathbb{Z}^n}$.
We also recall a basic fact that $C_c^{\fy}(\rn)\subset \mathscr{S}(\rn)\subset M^{p,q}_s$
for any $s\in \mathbb{R}, 0<p,q\leq \infty$.
\end{definition}

We say $X\hookrightarrow L^p_{0}(\rn)$, if for every $\va\in C_c^{\fy}(\rn)$ we have
\be
\|\va \ast f\|_{L^p(\rn)}\leq C\|f\|_X
\ee
for all $f\in X$, where the constant $C$ is only dependent on $\va$.

Following, we collect some basic embedding results of function spaces.
\begin{lemma}\label{lemma, ebd}
Let $0<p\leq 1, 0<q\leq \fy$, $s\in \mathbb{R}$.
\bn
\item For $X=h^p, M^{p,p}, B^{p,p}, F^{p,p}$,
we have $\|g\|_{L^p}\lesssim \|g\|_X$ for all measurable functions $g\in X$.
\item For $Y=M^{p,q}_s, \mathscr{F}M^{q,p}, B^{p,q}_s, F^{p,q}_s$, we have $Y\hookrightarrow L^p_{0}(\rn)$.
\en
\end{lemma}
\begin{proof}
  We first verify that $\|g\|_{L^p}\lesssim \|g\|_{h^p}$ for measurable functions $g\in h^p$.
  Take a $C_{c}^{\fy}$ function $\psi$ with $\int_{\mathbb{R}^n}\psi(x)dx=1$.
  We have
  \be
  \lim_{t\rightarrow 0}\psi_t\ast g=g\ \ \ \  a.e.\ \text{on}\  \rn.
  \ee
  Thus,
  \be
  \|g\|_{L^p}=\|\lim_{t\rightarrow 0}\psi_t\ast g\|_{L^p}
  \leq
\|\sup_{0<t<1}|\psi_t\ast g|\|_{L^p}=\|g\|_{h^p},
  \ee
 where we use the definition of $h^p$ in the last equality.

 For $X=M^{p,p}, B^{p,p}, F^{p,p}$, the inequality $\|g\|_{L^p}\lesssim \|g\|_X$ follows directly by
  the definition of function space and the triangle inequality, we only show the details for $B^{p,p}$:
  \be
  \begin{split}
    \|f\|_{L^p}
    =
    \bigg\|\sum_{j=0}^{\infty}\Delta_jf\bigg\|_{L^p}
    \leq
    \left(\sum_{j=0}^{\infty}\|\Delta_jf\|_{L^p}^p \right)^{1/p}
    =
    \|f\|_{B^{p,p}}=\|f\|_{F^{p,p}}.
  \end{split}
  \ee

 Next, we turn to the proof of statement (2).
 First, we deal with the case $Y=M^{p,q}_s$.
 Using the triangle inequality, we have
 \be
 \begin{split}
   \|\va\ast f\|_{L^p}
   =
   \bigg\|\sum_{k\in \mathbb{Z}^n}\Box_k (\va\ast f)\bigg\|_{L^p}
   \leq
   \left(\sum_{k\in \mathbb{Z}^n}\|\Box_k (\va\ast f)\|^p_{L^p}\right)^{1/p}.
 \end{split}
 \ee
Moreover, there exists a constant $c_n$ such that $\Box_k=\sum_{|l|\leq c_n}\Box_{k+l}\circ \Box_k$, then
 \be
 \begin{split}
 \|\Box_k (\va\ast f)\|_{L^p}
 =
 \bigg\|\sum_{|l|\leq c_n}\Box_{k+l}\va\ast \Box_k f\bigg\|_{L^p}
 \lesssim
 \sum_{|l|\leq c_n}\|\Box_{k+l}\va\|_{L^p}\|\Box_k f\|_{L^p}.
 \end{split}
 \ee
The above two estimates then yield that
\be
\begin{split}
  \|\va\ast f\|_{L^p}
  \lesssim &
  \left(\sum_{k\in \mathbb{Z}^n}\left(\sum_{|l|\leq c_n}\|\Box_{k+l}\va\|_{L^p}\|\Box_k f\|_{L^p}\right)^p\right)^{1/p}
  \\
  \leq &
  \sup_{k\in \mathbb{Z}^n}\langle k\rangle^s\|\Box_k f\|_{L^p}
  \left(\sum_{k\in \mathbb{Z}^n}\left(\langle k\rangle^{-s}\sum_{|l|\leq c_n}\|\Box_{k+l}\va\|_{L^p}\right)^p\right)^{1/p}
  \\
  \lesssim &
  \|f\|_{M^{p,\fy}_s}\|\va\|_{M^{p,p}_{-s}}\lesssim \|f\|_{M^{p,q}_s},
\end{split}
\ee
where we use the fact $\va\in M^{p,q}_{-s}$ and $M^{p,q}_s\subset M^{p,\fy}_s$.

For $Y=\mathscr{F}M^{q,p}$. By the conclusion for $Y=M^{p,\fy}$,
we have $\|\va\ast f\|_{L^p}\lesssim \|f\|_{M^{p,\fy}}.$
Note that $\mathscr{F}M^{q,p}$ is equal to the Wiener amalgam space $W^{p,q}$ (see\cite[pp.10]{GuoZhao2020JFA}),
and recall the embedding relations (see \cite[Lemma 2.5]{GuoWuYangZhao2017JFA}):
\be
     \mathscr{F}M^{q,p}\subset \mathscr{F}M^{\fy,p}=W^{p,\fy}\subset M^{p,\fy}.
\ee
The conclusion follows by
\be
\|\va\ast f\|_{L^p}
  \lesssim \|f\|_{M^{p,\fy}}\lesssim \|f\|_{\mathscr{F}M^{q,p}}.
\ee

For $Y=B^{p,q}_s$, there exists a constant $\d(p,q)>0$ such that $B^{p,q}_s\subset M^{p,q}_{s-\d(p,q)}$ (see \cite[Theorem 1.2]{GuoFanZhao2018SCM}). This and the conclusion for $Y=M^{p,q}_{s-\d(p,q)}$ imply that
\be
\|\va\ast f\|_{L^p}\lesssim \|f\|_{M^{p,q}_{s-\d(p,q)}}\lesssim \|f\|_{B^{p,q}_s}.
\ee

For $Y=F^{p,q}_s$, take a constant $\d>0$, then $F^{p,q}_{s}\subset B^{p,q}_{s-\d}$ (see \cite[pp.47]{Triebel2010}).
This and the conclusion for $Y=B^{p,q}_{s-\d}$ imply that
\be
\|\va\ast f\|_{L^p}\lesssim \|f\|_{B^{p,q}_{s-\d}}\lesssim \|f\|_{F^{p,q}_s}.
\ee


\end{proof}

\section{Main Theorems}
In this section, we give our main theorems and their proofs.
Suppose $X$ is a (Quasi-)Banach space with
translation invariant: $\|T_yf\|_X=\|f\|_X$,
where
$T_yf(x):= f(x-y)$ is the translation of $f$.

\begin{theorem}\label{Thm, lp}
  Let $0<p<1$, $\Phi\in L_{loc}^{1}(\rn\bs\{0\})$. Suppose $C_c^{\infty}(\rn)\subset X$ and $\|g\|_{L^p}\lesssim \|g\|_X$ for all measurable function $g\in X$.
  We have
  \be
  H_{\Phi}\ \text{is bounded on}\ X\Longleftrightarrow H_{\Phi}=0.
  \ee
\end{theorem}
\begin{proof}
  The ``if'' part is trivial. We only present the proof for the ``only if'' part.

  We first point out that $H_{\Phi}f$ is pointwise well-defined for any smooth function $f$ supported away from the origin.
  In fact, in this case we denote $E_{x}:=\{y: f\big(\frac{x}{|y|}\big)\neq 0\}$.
  Observe that for any fixed $x\in \rn$, $E_x$ is a bounded measurable set away from the origin.
  Recalling that $\Phi\in L_{loc}^{1}(\rn\bs\{0\})$, then the following integral is convergent:
  \be
  H_{\Phi}f(x)=\int_{\mathbb{R}^n}\Phi(y)f\left(\frac{x}{|y|}\right)dy=\int_{E_x}\Phi(y)f\left(\frac{x}{|y|}\right)dy.
  \ee
  Moreover, $H_{\Phi}f$ is a measurable function on $\rn$.
  Using the polar coordinates, we write
  \be
  \begin{split}
    H_{\Phi}f(x)
    = &\int_{\mathbb{R}^n}\Phi(y)f\left(\frac{x}{|y|}\right)dy
    \\
    = &
    \int_{0}^{\fy}\int_{\mathbb{S}^{n-1}}\Phi(ry')f(x/r)d\sigma(y')r^{n-1}dr
    = : \int_{0}^{\fy}\phi(r)f(x/r)dr,
  \end{split}
  \ee
  where
  \be
  \phi(r):=\int_{\mathbb{S}^{n-1}}\Phi(ry')r^{n-1}d\sigma(y').
  \ee
It follows by $\Phi\in L_{loc}^{1}(\rn\bs\{0\})$ that $\phi\in L_{loc}^{1}(\mathbb{R}^+)$.
Since $\phi$ is a measurable function on $(0,\fy)$, almost every point in $(0,\fy)$ is
a Lebesgue point.
Hence, it suffices to verify that if there exists a Lebesgue point $r_0>0$ such that $\phi(r_0)\neq 0$,
then $H_{\Phi}$ is unbounded.

Without loss of  generality, we assume that $r_0=1$ is a Lebesgue point of $\phi$, satisfying
$\phi(1)=1$. The proof for other cases is similar.

Taking $\theta=\frac{1-p}{2}$, we set
\be
A_j=[1-2^{-\theta j}, 1+2^{-\theta j}].
\ee

Take $g$ to be a nonnegative smooth function supported on $B(0,2)$, satisfying $g=1$ on $B(0,1)$.
Denote by $g_j(x):=g(x-2^je_0)$ the translation of $g$, where $e_0=(1,0,0,\cdots,0)$ be the unit vector on $\rn$.
For sufficiently large $j$, we have
\be
(2^j+2)(1-2^{-\theta j})+1<(2^j-2)(1+2^{-\theta j})-1.
\ee
Denote
\be
E_{j}:= \bigcup_{a\in \mathbb{R}: (2^j+2)(1-2^{-\theta j})+1\leq a\leq (2^j-2)(1+2^{-\theta j})-1}B(ae_0,1/2).
\ee
We have $|E_j|\sim 2^{(1-\theta)j}$.
For $x\in E_{j}$, we have
\be
(2^j+2)(1-2^{-\theta j})\leq |x|\leq (2^j-2)(1+2^{-\theta j}),\ \ \ |x|\sim 2^{j}.
\ee
This implies that
\ben\label{pf, 2}
\frac{|x|}{2^j-2}\leq 1+2^{-\theta j},\ \ \  \frac{|x|}{2^j+2}\geq 1-2^{-\theta j}.
\een
Recall $\text{supp}g\subset B(0,2)$.  For every $x\in E_{j}$,
set
\be
E^0_{j,x}:=\{r: g_j\big(\frac{x}{r}\big)\neq 0\},\ \ \ \ \ E^1_{j,x}:=\{r: g_j\big(\frac{x}{r}\big)=1\}.
\ee
By a direct calculation and \eqref{pf, 2}, we deduce that for $x\in E_j$,
\ben
E_{j,x}^1\subset E_{j,x}^0\subset \left(\frac{|x|}{2^j+2}, \frac{|x|}{2^j-2}\right)\subset A_j.
\een
Since $|x|\sim 2^{j}$ for $x\in E_j$, we have the upper estimate of $|E_{j,x}^0|$:
\be
|E^0_{j,x}|\leq \left|\frac{|x|}{2^j-2}-\frac{|x|}{2^j+2}\right|\sim 2^{-j}.
\ee
Next, we turn to the lower estimate of $|E^1_{j,x}|$.
For $x\in E_j$, write $x=ae_0+y$ for $y\in B(0,1/2)$, we have
\be
\begin{split}
  r\in E^1_{j,x}
  \Longleftarrow
  \left|\frac{x}{r}-2^je_0\right|\leq 1
  & \Longleftrightarrow
  \left|\frac{ae_0+y}{r}-2^je_0\right|\leq 1
  \\
  & \Longleftarrow
  \left|\frac{ae_0}{r}-2^je_0\right|+\left|\frac{y}{r}\right|\leq 1
  \\
  & \Longleftarrow
  \left|\frac{a}{r}-2^j\right|\leq \frac{1}{4},\ \ \ (\text{if}\ r\geq \frac{2}{3})
  \\
  & \Longleftrightarrow
  r\in \left[\frac{a}{2^j+1/4}, \frac{a}{2^j-1/4}\right].
\end{split}
\ee
Observe that
\be
\lim_{j\rightarrow \fy}\frac{a}{2^j+1/4}=\lim_{j\rightarrow \fy}\frac{a}{2^j-1/4}=1
\ee
for $a\in [(2^j+2)(1-2^{-\theta j})+1, (2^j-2)(1+2^{-\theta j})-1]$.
For sufficiently large $j$, we deduce that
$\left[\frac{a}{2^j+1/4}, \frac{a}{2^j-1/4}\right]\subset [2/3,\fy)$.
Hence,
\be
\left[\frac{a}{2^j+1/4}, \frac{a}{2^j-1/4}\right]\subset E_{j,x}^1.
\ee
The lower estimate of $|E_{j,x}^1|$ follows by
\be
|E_{j,x}^1|\geq \left|\frac{a}{2^j-1/4}-\frac{a}{2^j+1/4}\right|\sim 2^{-j}.
\ee
The combination of lower estimate of $|E_{j,x}^1|$ and upper estimate of $|E_{j,x}^0|$ yields that
\be
2^{-j}\lesssim |E^1_{j,x}|\leq |E^0_{j,x}|\lesssim 2^{-j},\ \ \ \text{or equivalent:}\ \ |E^1_{j,x}|\sim |E^0_{j,x}|\sim 2^{-j}.
\ee
Next, we divide the Hausdorff operator into main term and error term by
  \be
  \begin{split}
    H_{\Phi}g_j(x)
    &=
    \int_{0}^{\fy}\phi(r)g_j(x/r)dr
    =
    \int_{0}^{\fy}\phi(1)g_j(x/r)dr+\int_{0}^{\fy}(\phi(r)-\phi(1))g_j(x/r)dr
    \\
    &= :
    H_{\Phi}^Mg_j(x)+H_{\Phi}^Eg_j(x).
  \end{split}
  \ee
Let us first turn to the estimate of main term.
For $x\in E_j$, we have
  \be
  \begin{split}
    H_{\Phi}^Mg_j(x)
    =
    \int_{0}^{\fy}g_j(x/r)dr
    \geq
    \int_{E_{j,x}^1}g_j(x/r)dr
    = |E_{j,x}^1|\sim 2^{-j}.
  \end{split}
\ee
Recalling $|E_j|\sim 2^{(1-\theta)j}$ and $\theta=\frac{1-p}{2}$, we have following estimate of the main term:
  \ben\label{estimate of main term}
  \begin{split}
    \|H_{\Phi}^Mg_j\|^p_{L^p(E_j)}
    \gtrsim &
    2^{-jp}|E_j|\sim 2^{-jp}2^{(1-\theta)j}=2^{\frac{(1-p)j}{2}}.
  \end{split}
\een

On the other hand,
\ben\label{pf, 3}
\begin{split}
  \|H_{\Phi}^Eg_j\|^p_{L^p(E_j)}
  \leq &
  \|H_{\Phi}^Eg_j\|^p_{L^1(E_j)}|E_j|^{1-p}
  \\
  \leq &
  \left(\int_{E_j}\int_{0}^{\fy}|\phi(r)-\phi(1)|g_j(x/r)drdx\right)^p|E_j|^{1-p}
  \\
  = &
  \left(\int_{E_j}\int_{E_{j,x}^0}|\phi(r)-\phi(1)|g_j(x/r)drdx\right)^p|E_j|^{1-p}.
\end{split}
\een
Note that
\be
\{(x,r): x\in E_j, r\in E_{j,x}^0\}\subset \{(x,r): r\in A_j, x\in \rn\}.
\ee
We deduce that
\be
\begin{split}
  &\int_{E_j}\int_{E_{j,x}^0}|\phi(r)-\phi(1)|g_j(x/r)drdx
  \\
  \leq &
  \int_{A_j}|\phi(r)-\phi(1)|\int_{\rn}g_j(x/r)dxdr
  \\
  = &
  \|g\|_{L^1}\int_{A_j}|\phi(r)-\phi(1)|r^ndr
  \lesssim  \int_{A_j}|\phi(r)-\phi(1)|dr=\ep_j|A_j|\lesssim \ep_j2^{-\theta j},
\end{split}
\ee
where $\ep_j\rightarrow 0^+$ as $j\rightarrow \fy$.
Combining this with \eqref{pf, 3}, we have
\ben\label{estimate of error term}
\begin{split}
  \|H_{\Phi}^Eg_j\|^p_{L^p(E_j)}
  \lesssim
  \ep_j^p2^{-\theta pj}|E_j|^{1-p}
  \lesssim
  \ep_j^p2^{-\theta pj}2^{(1-\theta)(1-p)j}=\ep_j^p2^{\frac{(1-p)j}{2}}.
\end{split}
\een
By \eqref{estimate of main term} and \eqref{estimate of error term}, there exist two constants $C_1$ and $C_2$ such that
\be
\|H_{\Phi}^Mg_j\|^p_{L^p(E_j)}\geq C_12^{\frac{(1-p)j}{2}},\ \ \ \|H_{\Phi}^Eg_j\|^p_{L^p(E_j)}\leq C_2\ep_j^p2^{\frac{(1-p)j}{2}}.
\ee
For sufficiently large $j$ such that $C_2\ep_j^p\leq C_1/2$, we have
\be
\begin{split}
  \|H_{\Phi}g_j\|^p_{L^p}
  \geq &
  \|H_{\Phi}g_j\|^p_{L^p(E_j)}
  \\
  \geq &
  \|H_{\Phi}^Mg_j\|^p_{L^p(E_j)}-\|H_{\Phi}^Eg_j\|^p_{L^p(E_j)}
  \\
  \geq &
  C_12^{\frac{(1-p)j}{2}}-C_2\ep_j^p2^{\frac{(1-p)j}{2}}\geq (C_1/2)2^{\frac{(1-p)j}{2}}.
\end{split}
\ee
However, the boundedness of $H_{\Phi}$ yields that
\be
\|H_{\Phi}g_j\|^p_{L^p}
\lesssim \|H_{\Phi}g_j\|^p_X
\lesssim \|g_j\|^p_X=\|g\|^p_X,
\ee
which leads to a contradiction.
\end{proof}
Recall that all the spaces $L^p,h^p, M^{p,p}, B^{p,p}, F^{p,p}$ are translation invariant.
The following corollary is a direct conclusion of Lemma \ref{lemma, ebd} and Theorem \ref{Thm, lp}.
\begin{corollary}
 Let $0<p<1$, $\Phi\in L_{loc}^{1}(\rn\bs\{0\})$. We have
  \be
  H_{\Phi}\ \text{is bounded on}\ X\Longleftrightarrow H_{\Phi}=0,
  \ee
  where $X=L^p,h^p, M^{p,p}, B^{p,p}, F^{p,p}$.
\end{corollary}

For more general frequency decomposition spaces such as $X=B^{p,q}_s$, $F^{p,q}_s$ or $M^{p,q}_s$,
the embedding condition $\|g\|_{L^p}\lesssim \|g\|_X$ is no longer valid.
We establish following theorem with the help of the modified embedding (see Lemma \ref{lemma, ebd}).

\begin{theorem}\label{Thm, local lp}
  Let $0<p<1$, $\Phi\in L_{loc}^{1}(\rn\bs\{0\})$. Suppose $C_c^{\infty}(\rn)\subset X$ and $X\hookrightarrow L^p_{0}(\rn)$. We have
  \be
  H_{\Phi}\ \text{is bounded on}\ X\Longleftrightarrow H_{\Phi}=0.
  \ee
\end{theorem}
\begin{proof}
The ``if'' part is trivial. We focus on the proof for the ``only if'' part.
  As the proof of Theorem \ref{Thm, lp}, we write
  $H_{\Phi}f(x)=\int_{0}^{\fy}\phi(r)f(x/r)dr$,
  and assume
  $r_0=1$ is the Lebesgue point of $\phi$ with
$\phi(1)=1$. Let $g_j, A_j, E_j, E_{j,x}^0, E_{j,x}^1, H_{\Phi}^Mg_j, H_{\Phi}^Eg_j$ be as in the proof of Theorem \ref{Thm, lp}.
 Take a $C_c^{\fy}(\rn)$ nonnegative function $\va$ satisfying $\va(x)= 1$ on $B(0,1)$ and $\text{supp}\va\subset B(0,2)$.
 By the assumption we have
\be
\|\va \ast H_{\Phi}g_j\|_{L^p(\rn)}\lesssim \|H_{\Phi}g_j\|_{X}\lesssim \|g_j\|_X=\|g\|_X.
\ee
Set
\be
\Xi_j=\{x: B(x, 2)\subset E_j\}.
\ee
We have
\be
|\Xi_j|\sim |E_j|\sim 2^{(1-\theta)j}\ \ \ \ (j\rightarrow \fy).
\ee
Recall that $H_{\Phi}^Mg_j(x)\geq 0$ for $x\in \rn$ and
$H_{\Phi}^Mg_j(x)\gtrsim 2^{-j}$ for $x\in E_j$.
For $x\in \Xi_j$ we have
\be
\begin{split}
  \va \ast H_{\Phi}^Mg_j(x)
  = &
  \int_{\rn}\va(x-z)H_{\Phi}^Mg_j(z)dz
  \\
  \geq &
  \int_{B(x,1)}H_{\Phi}^Mg_j(z)dz
  \gtrsim \int_{B(x,1)}2^{-j}dz\sim 2^{-j}.
\end{split}
\ee
From this and the fact $|\Xi_j|\sim 2^{(1-\theta)j}$, we deduce that
\be
\begin{split}
  \|\va \ast H_{\Phi}^Mg_j\|^p_{L^p(\Xi_j)}
  \gtrsim &
  2^{-jp}|\Xi_j|\gtrsim 2^{-jp}2^{(1-\theta)j}  =2^{\frac{(1-p)j}{2}}.
\end{split}
\ee
On the other hand, observing that $E_{j,z}^0\subset A_j$ for $z\in B(x,2)$ with $x\in \Xi_j$,
\be
\begin{split}
  \|\va \ast H_{\Phi}^Eg_j\|^p_{L^p(\Xi_j)}
  \leq &
  \|\va \ast H_{\Phi}^Eg_j\|^p_{L^1(\Xi_j)}|\Xi_j|^{1-p}
  \\
  \leq &
  \left(\int_{\Xi_j}\int_{B(x,2)}\va(x-z)\int_{0}^{\fy}|\phi(r)-\phi(1)|g_j(z/r)drdzdx\right)^p|\Xi_j|^{1-p}
  \\
  = &
  \left(\int_{\Xi_j}\int_{B(x,2)}\va(x-z)\int_{E_{j,z}^0}|\phi(r)-\phi(1)|g_j(z/r)drdzdx\right)^p|\Xi_j|^{1-p}
  \\
  \leq &
  \left(\int_{A_j}|\phi(r)-\phi(1)|\int_{\rn}g_j(z/r)\int_{\rn}\va(x-z)dxdz dr\right)^p|\Xi_j|^{1-p}
  \\
  \lesssim &
  \left(\int_{A_j}|\phi(r)-\phi(1)|dr\right)^p|\Xi_j|^{1-p}\sim \ep_j^p|A_j|^p|\Xi_j|^{1-p}\sim \ep_j^p2^{\frac{(1-p)j}{2}},
\end{split}
\ee
where $\ep_j\rightarrow 0^+$ as $j\rightarrow \fy$.
Now, we have finished the estimates of main term $\|\va \ast H_{\Phi}^Mg_j\|^p_{L^p(\Xi_j)}$
and error term $\|\va \ast H_{\Phi}^Eg_j\|^p_{L^p(\Xi_j)}$, the remainder of this proof is the same
as that of Theorem \ref{Thm, lp}.
\end{proof}

Using Theorem \ref{Thm, local lp} and Lemma \ref{lemma, ebd}, we have following corollary.
\begin{corollary}
 Let $0<q\leq \fy, s\in \mathbb{R}$, $\Phi\in L_{loc}^{1}(\rn\bs\{0\})$. If $0<p<1$, we have
  \be
  H_{\Phi}\ \text{is bounded on}\ X \Longleftrightarrow H_{\Phi}=0,
  \ee
  where $X=F_s^{p,q}$, $B_s^{p,q}$ or $M^{p,q}_s$.
\end{corollary}

In particular, due to the time-frequency symmetry of modulation space, we have following corollary.

\begin{corollary}
 Let $0<p,q\leq \fy$. Let $\Phi$ be a measurable function satisfying
   \be
  \int_{B(0,1)}|y|^n\Phi(y)dy<\infty,\\
  \ \ \ \ \
  \int_{B(0,1)^c}\Phi(y)dy<\infty.
\ee
  If $0<p<1$ or $0<q<1$, we have
  \be
  H_{\Phi}\ \text{is bounded on}\ M^{p,q}\Longleftrightarrow H_{\Phi}=0.
  \ee
\end{corollary}
\begin{proof}
The ``if'' part is trivial. We focus on the proof for the ``only if'' part.

If $0<p<1$, we have $M^{p,q}\hookrightarrow L^p_{0}(\rn)$, then the conclusion follows by Theorem \ref{Thm, local lp}.

If $p\geq 1$, $0<q<1$, we will use the Fourier transform to exchange the time and frequency space.
  It follows by Remark 1.4 in \cite{ZhaoFanGuo2018AFA} that
\begin{equation*}
  \widehat{H_{\Phi}f}=\widetilde{H_{\Phi}}\widehat{f},
\end{equation*}
where
\begin{equation*}
  \widetilde{H_{\Phi}}f(x)=\int_{\mathbb{R}^n}\Phi(y)|y|^nf(|y|x)dy.
\end{equation*}
Thus,
\begin{equation*}
    \|H_{\Phi}f\|_{M^{p,q}}=\|\widehat{H_{\Phi}f}\|_{\mathscr{F}M^{p,q}}
    =\|\widetilde{H_{\Phi}}\widehat{f}\|_{\mathscr{F}M^{p,q}}.
  \end{equation*}
If $H_{\Phi}$ is bounded on $M_{p,q}$, we have
  \begin{equation*}
    \|\widetilde{H_{\Phi}}f\|_{\mathscr{F}M^{p,q}}
    =
    \|H_{\Phi}\check{f}\|_{M^{p,q}}
    \lesssim
    \|\check{f}\|_{M^{p,q}}
    =
    \|f\|_{\mathscr{F}M^{p,q}}.
  \end{equation*}
Write
\be
\begin{split}
  \widetilde{H_{\Phi}}f(x)
  = &
  \int_{\mathbb{R}^n}\Phi(y)|y|^nf(|y|x)dy
  \\
  = &
    \int_{0}^{\fy}\int_{\mathbb{S}^{n-1}}\Phi(ry')r^{n}f(rx)d\sigma(y')r^{n-1}dr
  \\
  = &
  \int_{0}^{\fy}\int_{\mathbb{S}^{n-1}}\Phi(y'/r)r^{-1-2n}f(x/r)d\sigma(y')dr
  = : \int_{0}^{\fy}\widetilde{\phi}(r)f(x/r)dr,
\end{split}
\ee
where
\be
\widetilde{\phi}(r)=\int_{\mathbb{S}^{n-1}}\Phi(y'/r)r^{-1-2n}d\sigma(y').
\ee
Observe $\widetilde{\phi}\in L^1_{loc}(\mathbb{R}^+)$
and
recall $\mathscr{F}M^{p,q}\hookrightarrow L^q_{0}(\rn)$ with $q\in (0,1)$.
By the same argument as in the proofs of Theorem \ref{Thm, lp} and \ref{Thm, local lp}, we conclude that $\widetilde{\phi}=0$ and complete this proof.
\end{proof}

\subsection*{Acknowledgements}
This work was partially supported by the National Natural Foundation of China (Nos. 11701112, 11771388, 11771358, 11671414, 11601456),
Zhejiang Provincial Natural Science Foundation of China (No. LY18A010015)
and
Natural Science Foundation of Fujian Province (Nos. 2017J01723, 2018J01430).

\bibliographystyle{abbrv}

\end{document}